\documentclass{amsart}
\usepackage[alphabetic,nobysame]{amsrefs}
\usepackage{amsmath,amsthm,amssymb,hyperref}
\usepackage[utf8]{inputenc}
\usepackage[english]{babel}

\theoremstyle{plain}
\newtheorem{theorem}{Theorem}[section]
\newtheorem{lemma}[theorem]{Lemma}
\newtheorem{corollary}[theorem]{Corollary}
\newtheorem{proposition}[theorem]{Proposition}

\theoremstyle{definition}
\newtheorem{definition}[theorem]{Definition}
\newtheorem{example}[theorem]{Example}

\title{On Hamiltonian minimality of isotropic non-homogeneous tori in $\mathbb{H}^n$ and $\mathbb{C} \mathrm{P}^{2n+1}$}

\author{Mikhail Ovcharenko}

\thanks{The author was supported by the State Maintenance Program for the Leading Scientific Schools of the Russian Federation (Grant NSh 5913.2018.1) and the Russian Foundation for Basic Research (Grant 18-01-00411)}

\DeclareMathOperator{\rk}{rk}
\DeclareMathOperator{\pr}{Pr}
\DeclareMathOperator{\diver}{div}

\begin{document}

\maketitle

\begin{abstract}
  We construct a family of flat isotropic non-homogeneous tori in $\mathbb{H}^n$ and $\mathbb{C} \mathrm{P}^{2n+1}$ and find necessary and sufficient conditions for their Hamiltonian minimality.
\end{abstract}

\section{Introduction}

A submanifold of real dimension $k$ in a K\"ahler manifold of complex dimension $n \geq k$ is called \emph{isotropic} if the K\"ahler form of the manifold vanishes on it. For $k = n$ the submanifold is called \emph{Lagrangian}. An isotropic submanifold in a K\"ahler manifold is called \emph{Hamiltonian-minimal} (shortly, \emph{H-minimal}) if the variations of its volume along the Hamiltonian fields with compact support are zero. In particular, any minimal isotropic submanifold is H-minimal.

The notion of H-minimality was introducted in the paper \cite{Oh93} in connection with the Arnold conjecture on the number of fixed points of a Hamiltonian symplectomorphism. The examples of H-minimal Lagrangian surfaces in $\mathbb{C}^2$ and $\mathbb{C} \mathrm{P}^2$ were found in \cite{CU98, HR00, HR02} and \cite{Mir03, Ma05}. A family of H-minimal Lagrangian submanifolds in $\mathbb{C}^n$ and $\mathbb{C} \mathrm{P}^n$ arising from intersections of real quadrics was considered in \cite{Mir04, MP13}. In \cite{Sha91} and \cite{Yer18} minimal Lagrangian and isotropic tori in $\mathbb{C} \mathrm{P}^2$ and $\mathbb{C} \mathrm{P}^3$ were studied in connection with the soliton equations: as a smooth periodic solutions of the Tzizeica equation and smooth periodic solutions of the sinh-Gordon equation accordingly. Necessary and sufficient conditions for H-minimality of isotropic homogeneous tori in $\mathbb{C}^n$ and $\mathbb{C} \mathrm{P}^n$ were found in \cite{Ovc18}.

Let us recall that a Riemannian manifold $M$ of real dimension $4n$ endowed with three complex structures $I,J,K$ is called \emph{hyperk\"ahler} if the following holds:
\begin{itemize}
\item $M$ is K\"ahler with respect to these structures;
\item $I,J,K$ satisfy the relation $I \circ J = -J \circ I = K$.
\end{itemize}
Let $(\cdot,\cdot)$ be the Riemannian metric on $M$. We will denote by
\[
  \omega_I = (I(\cdot), \cdot), \quad \omega_J = (J(\cdot), \cdot), \quad
  \omega_K = (K(\cdot), \cdot)
\]
the corresponding K\"ahler forms on $M$.

\begin{definition}
  We will call a submanifold of $N$ in $M$ \emph{$\omega_I$-isotropic} if $\omega_{I|N} = 0$. One can define $\omega_J$- and $\omega_K$-isotropic submanifolds in $M$ in the same way.
\end{definition}

Let us formulate the main results.

Consider a $(l+m)$-dimensional torus in $\mathbb{H}^n$ ($l,m \leq n$) defined by the mapping
\begin{gather*}
  \psi: \mathbb{R}^l \times \mathbb{R}^m \rightarrow \mathbb{H}^n, \quad
  \psi(x,y) = (r_1 e^{j(e_1,x)} e^{k(f_1,y)}, \dots,
  r_n e^{j(e_n,x)} e^{k(f_n,y)}), \\
  x \in \mathbb{R}^l, \quad y \in \mathbb{R}^m, \quad r_d > 0, \quad
  e_d \in \mathbb{R}^l, \quad f_d \in \mathbb{R}^m, \quad d = 1,\dots,n.
\end{gather*}
There $(\cdot, \cdot)$ is the standard dot product, the vectors $e_1,\dots,e_n$ and $f_1,\dots,f_n$ must generate lattices of maximal rank in $\mathbb{R}^l$ and $\mathbb{R}^m$ accordingly. Let us recall that by Euler's identity for quaternions we have
\[
  e^{j(e_d,x)} = \cos(e_d,x) + \sin(e_d,x) j, \quad
  e^{k(f_d,y)} = \cos(f_d,y) + \sin(f_d,y) k.
\]
We will denote this torus by $T^{l,m}$.

Notice that the matrix of the metric tensor of the torus $T^{l,m}$ in the coordinates $(x,y)$ is constant and of block form. Actually, we have $(\psi_{x_p}, \psi_{y_q}) = 0$ for all $p = 1,\dots,l$ and $q = 1,\dots,m$.

The following statement holds (see the proof in Section \ref{sec:Hn}).

\begin{proposition}\label{prop:isotropic}
  The torus $T^{l,m} \subset \mathbb{H}^n$ is $\omega_J$- and $\omega_K$-isotropic, but not $\omega_I$-isotropic.
\end{proposition}

We will denote the blocks of the matrix of the metric tensor of $T^{l,m}$ by
\begin{gather*}
  G_1 = (\psi_{x_p}, \psi_{x_q})_{pq}, \quad
  G_2 = (\psi_{y_r}, \psi_{y_s})_{rs}, \\
  p,q = 1,\dots,l; \quad r,s = 1,\dots,m.
\end{gather*}

Let us also introduce the vectors
\begin{gather*}
  e'_d = (e_d; \Vert e_d \Vert^2_{G_1^{-1}}
  + \Vert f_d \Vert^2_{G_2^{-1}}) \in \mathbb{R}^{l+1}, \quad
  \Vert e_d \Vert_{G_1^{-1}}^2 = (G_1^{-1}e_d, e_d), \\
  f'_d = (f_d; \Vert e_d \Vert^2_{G_1^{-1}}
  + \Vert f_d \Vert^2_{G_2^{-1}}) \in \mathbb{R}^{m+1}, \quad
  \Vert f_d \Vert_{G_2^{-1}}^2 = (G_2^{-1}f_d, f_d).
\end{gather*}

We have

\begin{theorem}\label{thm:Hn}
  The torus $T^{l,m} \subset \mathbb{H}^n$ is H-minimal
  \begin{itemize}
  \item with respect to $\omega_J$ --- if and only if
    $\rk_{\mathbb{Z}} \langle e'_1, \dots, e'_n \rangle = l$;
  \item with respect to $\omega_K$ --- if and only if
    $\rk_{\mathbb{Z}} \langle f'_1, \dots, f'_n \rangle = m$.
  \end{itemize}
\end{theorem}

\begin{corollary}
  The torus $T^{n,n} \subset \mathbb{H}^n$ is H-minimal with respect to $\omega_J$ and $\omega_K$ simultaneously.
\end{corollary}

Let us assume that the torus $T^{l,m} \subset \mathbb{H}^{n+1}$ is contained in the unit sphere $S^{4n+3}(1) \subset \mathbb{H}^{n+1}$, i.e., $\sum r_d^2 = 1$. For any $v = I,J,K$ we will consider the Hopf mapping
\[
  \pi_v: S^{4n+3}(1) \rightarrow \mathbb{C} \mathrm{P}^{2n+1}
\]
induced by the left action of the group $\{e^{vt}, t \in\mathbb{R}\}$. Notice that for any $v$ we have the Fubini~--- Study form $\Omega_v$ on $\mathbb{C} \mathrm{P}^{2n+1}$ induced by the form $\omega_v$ on $\mathbb{H}^{n+1}$.

\begin{definition}
  We will call a submanifold $N \subset S^{4n+3}(1)$ \emph{$\pi_I$-horizontal} if $N$ is orthogonal to the fibers of the fibration $\pi_I$. One can define $\pi_J$- and $\pi_K$-horizontal submanifolds in $S^{4n+3}(1)$ in the same way.
\end{definition}

Let us introduce the vectors
\begin{gather*}
  e''_d = (e_d; \Vert e_d \Vert^2_{G_1^{-1}} + \Vert f_d
  \Vert^2_{G_2^{-1}} - (l+m)) \in \mathbb{R}^{l+1}, \\
  f''_d = (f_d; \Vert e_d \Vert^2_{G_1^{-1}} + \Vert f_d
  \Vert^2_{G_2^{-1}} - (l+m)) \in \mathbb{R}^{m+1}.
\end{gather*}

We have

\begin{theorem}\label{thm:HPn}
  Let the torus $T^{l,m} \subset S^{4n+3}(1)$ be $\pi_J$-horizontal. The torus $\pi_J(T^{l,m})$ is H-minimal if and only if
  \[
    \rk_{\mathbb{Z}} \langle e''_1, \dots, e''_n \rangle = l.
  \]
  Similarly, let the torus $T^{l,m}$ be $\pi_K$-horizontal. The torus $\pi_K(T^{l,m})$ is H-minimal if and only if
  \[
    \rk_{\mathbb{Z}} \langle f''_1, \dots, f''_n \rangle = m.
  \]
\end{theorem}

\begin{corollary}
  Let the torus $T^{n,n} \subset S^{4n+3}(1)$ be $\pi_J$-horizontal. Then the torus $\pi_J(T^{n,n})$ is H-minimal.

  Similarly, let the torus $T^{n,n}$ be $\pi_K$-horizontal. Then the torus $\pi_K(T^{n,n})$ is H-minimal.
\end{corollary}

Let us notice that the tori $\pi_J(T^{1,1})$ and $\pi_K(T^{1,1})$ are \emph{minimal} submanifolds in $\mathbb{C} \mathrm{P}^3$ (see \cite{Yer19}).

The proof of Theorem \ref{thm:Hn} and \ref{thm:HPn} is given in Section \ref{sec:Hn} and \ref{sec:HPn} accordingly. There we also give explicit examples of tori under consideration.

The author is grateful to A.~Mironov for stating the problem and useful discussions.

\section{Proof of Theorem \ref{thm:Hn}}
\label{sec:Hn}

We will denote by
\[
  \langle x,y \rangle = \sum x_i \bar{y}_i,
  \quad x,y \in \mathbb{H}^n,
\]
the standard Hermitian inner product in $\mathbb{H}^n$. By definition we have
\[
  \langle x,y \rangle = (x,y) - \omega_I(x,y) i
  - \omega_J(x,y) j - \omega_K(x,y) k.
\]

\begin{proof}[Proof of Proposition \ref{prop:isotropic}]
  One can easily check that for any given tangent vectors $\xi$, $\eta$ on $T^{l,m}$ we have $\langle \xi, \eta \rangle \in \mathbb{C}$, i.e., $\omega_J(\xi, \eta) = \omega_K(\xi,\eta) = 0$.

  Nonetheless, the torus $T^{l,m}$ is not $\omega_I$-isotropic. Actually,
  \[
    \omega_I(\psi_{x_p}, \psi_{y_q})
    = e_{1p}f_{1q}r_1^2 + \dots + e_{np}f_{nq}r_n^2.
  \]
  The matrix of the metric tensor of $T^{l,m}$ in the coordinates $(x,y)$ is constant and of block form, so after proper linear change of coordinates on $\mathbb{R}^l$ and $\mathbb{R}^m$ we can choose $e_1$ and $f_1$ to be the first standard basis vector in $\mathbb{R}^l$ and $\mathbb{R}^m$ accordingly. Then we have $\omega_I(\psi_{x_1}, \psi_{y_1}) = r_1^2 \neq 0$.
\end{proof}

To prove Theorem \ref{thm:Hn} we are going to use the following Chen and Morvan criterion of H-minimality for compact isotropic submanifolds.

\begin{proposition}[see \cite{CM94}*{Proposition~3.5}]\label{prop:chen}
  Let $N$ be a compact isotropic submanifold of a K\"ahler manifold $(M,I)$ with the K\"ahler form $\omega$, $H$ be the mean curvature field of $N$. The submanifold $N$ is H-minimal if and only if $IH$ is a tangent vector field to $N$ and $\delta (i_H \omega) = 0$, where $i_H \omega = \omega (H,\cdot)$ is a 1-form on $N$, $\delta$ is the codifferential. \qed
\end{proposition}

Consequently, a compact $\omega_J$-isotropic submanifold $N$ of a hyperk\"ahler manifold $M$ is H-minimal with respect to $\omega_J$ if and only if $JH$ is a tangent vector field to $N$ and $\delta \omega_J(H,\cdot) = 0$ (the statement for $\omega_K$ is analogous).

\begin{lemma}\label{lem:mean-curvature-Hn}
  The mean curvature field of the torus $T^{l,m}$ in $\mathbb{H}^n$ has the following form:
  \[
    H = -\frac{1}{l+m}
    (\alpha_1 r_1 e^{j(e_1,x)} e^{k(f_1,y)}, \dots,
    \alpha_n r_n e^{j(e_n,x)} e^{k(f_n,y)}).
  \]
  There we put $\alpha_d = \Vert e_d \Vert^2_{G_1^{-1}} +
  \Vert f_d \Vert^2_{G_2^{-1}} = (G_1^{-1}e_d,e_d) + (G_2^{-1}f_d,f_d)$.
\end{lemma}

\begin{proof}
  The mean curvature field of $T^{l,m}$ equals (see, for example, \cite[Proposition~8]{Law80})
  \[
    H = \frac{1}{l+m} \Delta \psi,
  \]
  where $\Delta$ is the Laplace~--- Beltrami operator of the torus $T^{l,m}$ in the coordinates $(x,y)$. Since the matrices $G_1$ and $G_2$ are constant, we have
  \[
    \Delta = \sum_{p,q=1}^l G_1^{pq} \partial_{x_p}\partial_{x_q} +
    \sum_{r,s=1}^m G_2^{rs} \partial_{y_r}\partial_{y_s},
  \]
  where $G_1^{pq}$, $G_2^{rs}$ are the coefficients of the matrices $G_1^{-1}$ and $G_2^{-1}$ accordingly.

  Consequently, we obtain
  \begin{gather*}
    \sum_{p,q=1}^l G_1^{pq} \partial_{x_p} \partial_{x_q} e^{j(e_d,x)}
    = -(G_1^{-1} e_d,e_d) e^{j(e_d,x)}, \\
    \sum_{r,s=1}^m G_2^{rs} \partial_{y_r} \partial_{y_s} e^{k(f_d,y)}
    = -(G_2^{-1} f_d,f_d) e^{k(f_d,y)}
  \end{gather*}
  for all $d = 1,\dots,n$.
\end{proof}

\begin{lemma}\label{lem:mean-curvature-tangent}
  Let $H$ be the mean curvature field of the torus $T^{l,m}$. The vector field $JH$ is tangent to $T^{l,m}$ if and only if
  \[
    \rk_{\mathbb{Z}} \langle e'_1, \dots, e'_n \rangle = l.
  \]
\end{lemma}

\begin{proof}
  The vector field $JH$ is tangent to $T^{l,m}$ if and only if there exist real functions $\lambda_1(x,y), \dots, \lambda_l(x,y)$; $\eta_1(x,y), \dots, \eta_m(x,y)$ such that
  \begin{align*}
    H = -\frac{1}{l+m}
    (\alpha_1 r_1 e^{j(e_1,x)} e^{k(f_1,y)}, \dots,
    \alpha_n r_n e^{j(e_n,x)} e^{k(f_n,y)}) \\
    = j(\lambda_1 \partial_{x_1} \psi + \dots
    + \lambda_l \partial_{x_l} \psi + \eta_1 \partial_{y_1} \psi
    + \dots + \eta_m \partial_{y_m} \psi).
  \end{align*}
  This is equivalent to the identities
  \[
    \frac{\alpha_d}{l+m} = (\lambda,e_d) - i(\eta,f_d),
    \quad d = 1,\dots,n,
  \]
  where we put $\lambda = (\lambda_1, \dots, \lambda_l)$, $\eta = (\eta_1, \dots, \eta_m)$. It is clear that $\eta$ should be zero. We have
  \[
    \frac{\alpha_d}{l+m} = (\lambda,e_d), \quad d = 1,\dots,n.
  \]
  This overdetermined system of linear equations has a solution $\lambda$ (where $\lambda$ does not depend on $x$ and $y$) if and only if
  \[
    \rk_{\mathbb{R}} \langle e'_1, \dots, e'_n \rangle
    = \rk_{\mathbb{R}} \langle e_1, \dots, e_n \rangle = l.
  \]
  Recall that the vectors $e_1,\dots,e_n$ define a lattice of rank $l$ in $\mathbb{R}^l$, hence
  \[
    \rk_{\mathbb{R}} \langle e_1, \dots, e_n \rangle
    = \rk_{\mathbb{Z}} \langle e_1, \dots, e_n \rangle.
  \]
  On the other hand, we have $e_d' = (e_d; \Vert e_d \Vert^2_{G_1^{-1}}
  + \Vert f_d \Vert^2_{G_2^{-1}})$, so
  \[
    \rk_{\mathbb{R}} \langle e'_1, \dots, e'_n \rangle =
    \rk_{\mathbb{Z}} \langle e'_1, \dots,  e'_n \rangle = l.
  \]
\end{proof}

\begin{lemma}\label{lem:mean-curvature-coclosed}
  Suppose that the vector field $JH$ is tangent to the torus $T^{l,m}$. Then $\delta \omega_J(H,\cdot) = 0$.
\end{lemma}

\begin{proof}
  Notice that there is linear change of coordinates $z = Ax$ and $w = By$ on $\mathbb{R}^l$ and $\mathbb{R}^m$ accordingly such that
  \[
    (\partial_{z_p} \psi, \partial_{z_q}\psi) = \delta_{pq}, \quad
    (\partial_{w_r} \psi, \partial_{w_s}\psi) = \delta_{rs},
  \]
  which follows from the Gram~--- Schmidt process. For simplifying the proof, it is convinient to pass to the coordinates $(z,w)$.

  Recall that $\omega_J(H,\cdot) = (JH,\cdot)$. We have (see, for example, \cite{Bes08})
  \[
    \delta \omega_J(H,\cdot) = \diver JH =
    \sum_{p=1}^l (\nabla_{\psi_{z_p}} JH, \psi_{z_p}) +
    \sum_{q=1}^m (\nabla_{\psi_{w_q}} JH, \psi_{w_q}),
  \]
  where $\nabla$ is a connection on the torus $T^{l,m}$ compatible with the induced metric. Let us show that $\nabla_{\psi_{z_p}} JH = \nabla_{\psi_{w_q}} JH = 0$ for any $p = 1,\dots,l$ and $q = 1,\dots,m$.

  Actually,
  \begin{align*}
    \nabla_{\psi_{z_p}} JH = \pr(\partial_{z_p} JH) \\
    = -\frac{1}{l+m} \pr
    (\partial_{z_p} j (\alpha_1 r_1 e^{j(e_1,z)} e^{k(f_1,w)},
    \dots, \alpha_n r_n e^{j(e_n,z)} e^{k(f_n,w)})) \\
    = \frac{1}{l+m} \pr
    (e_{1p} \alpha_1 r_1 e^{j(e_1,z)} e^{k(f_1,w)}, \dots,
    e_{np} \alpha_n r_n e^{j(e_n,z)} e^{k(f_n,w)}),
  \end{align*}
  where $\pr$ is the orthogonal projection of the vector field to the torus. Furthermore,
  \begin{gather*}
    \partial_{z_r} \psi = j (e_{1r}r_1 e^{j(e_1,x)} e^{k(f_1,y)},
    \dots, e_{nr} r_n e^{j(e_n,x)} e^{k(f_n,y)}), \\
    \partial_{w_s} \psi = k (f_{1s}r_1 e^{j(e_1,x)} e^{k(f_1,y)},
    \dots, f_{ns} r_n e^{j(e_n,x)} e^{k(f_n,y)}).
  \end{gather*}
  Consequently,
  \begin{gather*}
    \langle \partial_{z_p} JH, \partial_{z_r} \psi \rangle
    = -\frac{j}{l+m}(e_{1p}e_{1r} \alpha_1 r_1^2
    + \dots + e_{np}e_{nr} \alpha_n r_n^2), \\
    \langle \partial_{z_p} JH, \partial_{w_s} \psi \rangle
    = -\frac{k}{l+m}(e_{1p}f_{1s} \alpha_1 r_1^2
    + \dots + e_{np}f_{ns} \alpha_n r_n^2),  
  \end{gather*}
  then
  \[
    (\partial_{z_p} JH, \partial_{z_r} \psi) =
    (\partial_{z_p} JH, \partial_{w_s} \psi)= 0
  \]
  for any $p,r = 1,\dots,l$ and $s = 1,\dots,m$.

  One can prove in the same way that $\nabla_{\psi_{w_q}} JH = 0$ for any $q = 1,\dots,m$.

  We obtain $\delta \omega_J(H,\cdot) = 0$.
\end{proof}

The counterpart of Lemmas \ref{lem:mean-curvature-tangent} and \ref{lem:mean-curvature-coclosed} for $\omega_K$ can be proved in the same way.

Theorem \ref{thm:Hn} is proved. Let us consider some examples.

\begin{example}
  Let us consider the torus $T^{n,m}$ in $\mathbb{H}^n$. Since $l = n$, the first condition of Theorem \ref{thm:Hn}
  \[
    \rk_{\mathbb{Z}} \langle e'_1, \dots, e'_n \rangle = l
  \]
  holds trivially, and the torus $T^{n,m}$ is H-minimal with respect to $\omega_J$.

  Let $f_1,\dots,f_m$ be an arbitrary basis $\mathbb{R}^m$, and
  \[
    f_{m+1},\dots,f_n \in \{f_1,\dots,f_m\}, \;
    r_{m+1},\dots,r_n \in \{r_1,\dots,r_m\},
  \]
  where $f_p = f_q$ if and only if $r_p = r_q$ for any $p = 1,\dots,m$; $q = m+1,\dots,n$. One can directly check that $G_1 = \text{diag}(r_1^2,\dots,r_n^2)$ and
  $\Vert e_d \Vert_{G_1^{-1}}^2 = 1/r_d^2$. Then $f'_{m+1},\dots,f'_n \in \{f'_1,\dots,f'_m\}$ and
  \[
    \rk_{\mathbb{Z}} \langle f'_1, \dots, f'_n \rangle = m.
  \]
  Consequently, $T^{n,m}$ is also H-minimal with respect to $\omega_K$ by Theorem \ref{thm:Hn}.
\end{example}

Let us notice that the mapping
\[
  \psi: \mathbb{R}^l \times \mathbb{R}^m \rightarrow \mathbb{H}^n, \;
  \psi(x,y) = (r_1 e^{j(e_1,x)} e^{k(f_1,y)}, \dots,
  r_n e^{j(e_n,x)} e^{k(f_n,y)})
\]
defines two projections $\psi(x,0)$ and $\psi(0,y)$ on the homogeneous tori $T^{l,0}$ and $T^{0,m}$ in $\mathbb{C}^n$ accordingly. Necessary and sufficient conditions for their H-minimality are given by \cite[Theorem~1]{Ovc18}.

\begin{example}
  The torus $T^{l,l} \subset \mathbb{H}^n$ defined by the mapping
  \[
    \psi(x,y) = (r_1 e^{j(e_1,x)} e^{k(e_1,y)},
    \dots, r_n e^{j(e_n,x)} e^{k(e_n,y)})
  \]
  is H-minimal with respect to $\omega_J$ and $\omega_K$ simultaneously if and only if the homogeneous torus $T^{l,0}$ is H-minimal in $\mathbb{C}^n$.
\end{example}

\begin{example}
  The torus $T^{2,2} \subset \mathbb{H}^4$ defined by the mapping
  \[
    \psi(x,y) = (2e^{jx_1}e^{ky_1}, e^{jx_2}e^{ky_2},
    \frac{1}{2} e^{j(2x_1+2x_2)}e^{ky_1}, e^{j(x_1+2x_2)}e^{ky_2})
  \]
  is not H-minimal with respect to neither $\omega_J$ nor $\omega_K$. Still, the homogeneous tori $T^{2,0}$ and $T^{0,2}$ are H-minimal in $\mathbb{C}^4$.
\end{example}

\begin{example}
  The torus $T^{3,3} \subset \mathbb{H}^4$ defined by the mapping
  \[
    \psi(x,y) = (e^{jx_1} e^{ky_1}, e^{jx_2} e^{ky_2},
    e^{jx_3} e^{ky_3}, \frac{1}{\sqrt{2}} e^{j2x_1} e^{k2y_2})
  \]
  is H-minimal with respect to $\omega_J$ and $\omega_K$ simultaneously. Still, the homogeneous tori $T^{3,0}$ and $T^{0,3}$ \emph{are not} H-minimal in $\mathbb{C}^4$.
\end{example}

\section{Proof of Theorem \ref{thm:HPn}}
\label{sec:HPn}

Let us consider the torus $T^{l,m} \subset S^{4n+3}(1)$ in the unit sphere, i.e., $\sum r_d^2 = 1$. One can easily check that conditions of $\pi_J$- and $\pi_K$-horizontality of the torus $T^{l,m}$ are equivalent to the identities $\sum r_d^2 e_d = 0$ and $\sum r_d^2 f_d = 0$ accordingly.

To prove Theorem \ref{thm:HPn} we are going to use the following statement, which follows from Chen and Morvan criterion of H-minimality (see Proposition \ref{prop:chen}).

\begin{lemma}[see \cite{Ovc18}*{Lemma~6}]
  Let $I \subset S^{2n+1}(1)$ be a horizontal lifting of an isotropic submanifold in $\mathbb{C} \mathrm{P}^n$. Let us denote by $\hat{H}$ the mean curvature field of $I$ in $S^{2n+1}(1)$. Then $\pi(I)$ is H-minimal in $\mathbb{C} \mathrm{P}^n$ if and only if $J \hat{H}$ is a tangent vector field to $I$ and $\delta \alpha_{\hat{H}} = 0$. \qed
\end{lemma}

Let $N \subset S^{4n+3}(1)$ be a compact $\pi_J$-horizontal $\omega_J$-isotropic submanifold in $\mathbb{H}^{n+1}$, $\hat{H}$ be the mean curvature field of $N$ in $S^{4n+3}(1)$. One can derive from Lemma that the submanifold $\pi_J(N)$ is H-minimal in $\mathbb{C} \mathrm{P}^{2n+1}$ if and only if $J \hat{H}$ is a tangent vector field to $N$ and $\delta \omega_J (\hat{H},\cdot) = 0$ (the statement for $\omega_K$ is analogous).

Now we are going to prove Theorem \ref{thm:HPn}. First of all, we need to find the mean curvature field of the torus $T^{l,m}$ in the sphere $S^{4n-1}(r)$.

\begin{lemma}
  The mean curvature field of the torus $T^{l,m}$ in $S^{4n-1}(r)$ has the following form:
  \[
    \hat{H} =
    \left ( \left ( \frac{1}{r^2} - \frac{\alpha_1}{l+m} \right )
    r_1 e^{j(e_1,x)} e^{k(f_1,y)}, \dots,
    \left  ( \frac{1}{r^2} - \frac{\alpha_n}{l+m} \right )
    r_n e^{j(e_n,x)} e^{k(f_n,y)} \right ).
  \]
  There we put $\alpha_d = \Vert e_d \Vert^2_{G_1^{-1}} + \Vert f_d
  \Vert^2_{G_2^{-1}} = (G_1^{-1}e_d,e_d) + (G_2^{-1}f_d,f_d)$.
\end{lemma}

\begin{proof}
  We have the embeddings $T^{l,m} \subset S^{4n-1}(r) \subset \mathbb{H}^n$. Then
  \[
    \hat{H} = \text{P}(H),
  \]
  where $\text{P}$ is the orthogonal projection of the vector field $H$ to $S^{4n-1}(r)$ (see \cite{Law80}). Consequently,
  \[
    \hat{H} = H - \frac{(H,\psi)}{r^2} \psi.
  \]
  Let us show that $(H, \psi) = -1$. To this end, it is convinient to pass to the coordinates $(z,w)$ on $T^{l,m}$ in which
  \[
    (\partial_{z_p} \psi, \partial_{z_q}\psi) = \delta_{pq}, \;
    (\partial_{w_r} \psi, \partial_{w_s}\psi) = \delta_{rs}
  \]
  (see the proof of Lemma \ref{lem:mean-curvature-coclosed}). We have
  \begin{align*}
    (H,\psi) =
    -\frac{1}{l+m}(\alpha_1 r_1^2 + \dots + \alpha_n r_n^2) \\
    = -\frac{1}{l+m} \left ( \sum_{p=1}^l (\partial_{z_p} \psi,
    \partial_{z_p} \psi) + \sum_{q=1}^m (\partial_{w_q} \psi,
    \partial_{w_q} \psi) \right ) = -1.
  \end{align*}
  Then the statement follows from Lemma \ref{lem:mean-curvature-Hn}.
\end{proof}

\begin{corollary}
  The torus $T^{l,m}$ is minimal in $S^{4n-1}(r)$ if and only if
  \[
    \Vert e_d \Vert^2_{G_1^{-1}} + \Vert f_d \Vert^2_{G_2^{-1}} =
    \frac{l+m}{r^2}.
  \]
\end{corollary}

\begin{proposition}\label{prop:minimality}
  Let $T^{l,m}$ be the torus defined by the mapping
  \[
    \psi: \mathbb{R}^l \times \mathbb{R}^m \rightarrow \mathbb{H}^n,
    \; \psi(x,y) = (r_1 e^{j(e_1,x)} e^{k(f_1,y)}, \dots,
    r_n e^{j(e_n,x)} e^{k(f_n,y)}),
  \]
  and let $T^{l,0}$ and $T^{0,m}$ be the homogeneous tori in $\mathbb{C}^n$ defined by projections $\psi(x,0)$ and $\psi(0,y)$ accordingly.

  If the tori $T^{l,0}$ and $T^{0,m}$ are minimal in $S^{2n-1}(r)$, then the torus $T^{l,m}$ is minimal in $S^{4n-1}(r)$.
\end{proposition}

\begin{proof}
  By \cite[Corollary~1]{Ovc18} the torus $T^{l,0}$ is minimal in $S^{2n-1}(r)$ if and only if
  \[
    \Vert e_d \Vert^2_{G_1^{-1}} = \frac{l}{r^2},
  \]
  where $G_1$ is the matrix of the metric tensor of the torus $T^{l,0}$.

  Similarly, the torus $T^{0,m}$ is minimal in $S^{2n-1}(r)$ if and only if
  \[
    \Vert f_d \Vert^2_{G_2^{-1}} = \frac{m}{r^2},
  \]
  where $G_2$ is the matrix of the metric tensor of the torus $T^{0,m}$.
 
  At last, notice that $G_1$ and $G_2$ are precisely the blocks of the matrix of the metric tensor of $T^{l,m}$.
\end{proof}

\begin{example}
  The torus $T^{3,3}$ in $\mathbb{H}^4$ defined by the mapping
  \[
    \psi(x,y) = \left (e^{jx_1} e^{ky_1},
    \sqrt{\frac{3}{2}}e^{jx_2} e^{ky_2},
    \sqrt{\frac{3}{2}} e^{jx_3} e^{ky_3},
    e^{-j(x_1+x_2+x_3)} e^{-ky_1} \right )
  \]
  is minimal in $S^{17}(\sqrt{5}) \subset \mathbb{H}^4$. Still, the homogeneous tori $T^{3,0}$ and $T^{0,3}$ \emph{are not} minimal in $S^7(\sqrt{5}) \subset \mathbb{C}^4$ (see \cite[Corollary~1]{Ovc18}).
\end{example}

\begin{lemma}\label{lem:mean-curvature-HPn-tangent}
  Let $\hat{H}$ be the mean curvature field of the torus $T^{l,m}$ in $S^{4n-1}(1)$. The vector field $J \hat{H}$ is tangent to $T^{l,m}$ if and only if
  \[
    \rk_{\mathbb{Z}} \langle e''_1, \dots, e''_n \rangle = l.
  \]
\end{lemma}

\begin{proof}
  The proof is similar to the proof of Lemma \ref{lem:mean-curvature-tangent}.
\end{proof}

\begin{lemma}\label{lem:mean-curvature-HPn-coclosed}
  Suppose that the vector field $J \hat{H}$ is tangent to the torus $T^{l,m}$. Then $\delta \omega_J(\hat{H}, \cdot) = 0$.
\end{lemma}

\begin{proof}
  Let us notice that there is linear change of coordinages $z = Ax$ and $w = By$ on $\mathbb{R}^l$ and $\mathbb{R}^m$ accordingly such that
  \[
    (\partial_{z_p} \psi, \partial_{z_q} \psi) = \delta_{pq}, \;
    (\partial_{w_r} \psi, \partial_{w_s} \psi) = \delta_{rs},
  \]
  which follows from the Gram~--- Schmidt process. For simplifying the proof, it is convinient to pass to the coordinates $(z,w)$.

  Recall that $\omega_J(\hat{H}, \cdot) = (J \hat{H}, \cdot)$. We have (see, for example, \cite{Bes08})
  \[
    \delta \omega_J(\hat{H}, \cdot) = \diver J \hat{H}
    = \sum_{p=1}^l (\nabla_{\psi_{z_p}} J \hat{H}, \psi_{z_p})
    + \sum_{q=1}^m (\nabla_{\psi_{w_q}} J \hat{H}, \psi_{w_q}),
  \]
  where $\nabla$ is a connection on the torus $T^{l,m}$ compatible with the induced metric. Let us show that $\nabla_{\psi_{z_p}} J \hat{H} = \nabla_{\psi_{w_q}} J \hat{H} = 0$ for any $p = 1,\dots,l$ and $q = 1,\dots,m$.

  Actually,
  \begin{align*}
    \nabla_{\psi_{z_p}} J \hat{H}
    = \pr(\partial_{z_p} J \hat{H}) \\
    = \pr(\partial_{z_p} j
    (\beta_1 r_1 e^{j(e_1,z)} e^{k(f_1,w)}, \dots,
    \beta_n r_n e^{j(e_n,z)} e^{k(f_n,w)})) \\
    = -\pr (e_{1p} \beta_1 r_1 e^{j(e_1,z)} e^{k(f_1,w)},
    \dots, e_{np} \beta_n r_n e^{j(e_n,z)} e^{k(f_n,w)}),
  \end{align*}
  where $\pr$ is the orthogonal projection of the vector field to the torus. We also put $\beta_d = 1/r^2 - \alpha_d / (l+m)$. Furthermore,
  \begin{gather*}
    \partial_{z_r} \psi = j (e_{1r}r_1 e^{j(e_1,x)} e^{k(f_1,y)},
    \dots, e_{nr} r_n e^{j(e_n,x)} e^{k(f_n,y)}), \\
    \partial_{w_s} \psi = k (f_{1s}r_1 e^{j(e_1,x)} e^{k(f_1,y)},
    \dots, f_{ns} r_n e^{j(e_n,x)} e^{k(f_n,y)}).
  \end{gather*}
  Consequently,
  \begin{gather*}
    \langle \partial_{z_p} J \hat{H}, \partial_{z_r} \psi \rangle
    = j(e_{1p}e_{1r} \beta_1 r_1^2 + \dots
    + e_{np}e_{nr} \beta_n r_n^2), \\
    \langle \partial_{z_p} J \hat{H}, \partial_{w_s} \psi \rangle
    = k(e_{1p}f_{1s} \beta_1 r_1^2 + \dots
    + e_{np}f_{ns} \beta_n r_n^2),
  \end{gather*}
  then
  \[
    (\partial_{z_p} J \hat{H}, \partial_{z_r} \psi) =
    (\partial_{z_p} J \hat{H}, \partial_{w_s} \psi) = 0
  \]
  for any $p,r = 1,\dots,l$ and $s = 1,\dots,m$.

  One can prove in the same way that $\nabla_{\psi_{w_q}} J \hat{H} = 0$ for any $q = 1,\dots,m$.

  We obtain $\delta \omega_J(\hat{H}, \cdot) = 0$.
\end{proof}

The counterpart of Lemmas \ref{lem:mean-curvature-HPn-tangent} and \ref{lem:mean-curvature-HPn-coclosed} for $\omega_K$ can be proved in the same way.

Theorem \ref{thm:HPn} is proved. Let us consider some examples.

Recall that the mapping
\[
  \psi: \mathbb{R}^l \times \mathbb{R}^m \rightarrow \mathbb{H}^n, \;
  \psi(x,y) = (r_1 e^{j(e_1,x)} e^{k(f_1,y)}, \dots,
  r_n e^{j(e_n,x)} e^{k(f_n,y)})
\]
defines two projections $\psi(x,0)$ and $\psi(0,y)$ on the homogeneous tori $T^{l,0}$ and $T^{0,m}$ in $\mathbb{C}^n$ accordingly. One can easily check that conditions of $\pi_J$- and $\pi_K$-horizontality of the torus $T^{l,m}$
\[
  \sum r_d^2 e_d = 0, \quad \sum r_d^2 f_d = 0
\]
are precisely the conditions of horizontality of the tori $T^{l,0}$ and $T^{0,m}$ accordingly.

Let us denote by
\[
  \pi: S^{2n+1}(1) \rightarrow \mathbb{C} \mathrm{P}^n
\]
the Hopf fibration. Necessary and sufficient conditions for H-minimality of the tori $\pi(T^{l,0})$ and $\pi(T^{0,m})$ in $\mathbb{C} \mathrm{P}^{n-1}$ are given by \cite[Theorem~2]{Ovc18}.

\begin{example}
  Let $T^{l,l} \subset S^{4n+3}(1)$ be the $\pi_J$- and $\pi_K$-horizontal torus in $\mathbb{H}^{n+1}$ defined by the mapping
  \[
    \psi(x,y) = (r_1 e^{j(e_1,x)} e^{k(e_1,y)}, \dots,
    r_{n+1} e^{j(e_{n+1},x)} e^{k(e_{n+1},y)}).
  \]
  The tori $\pi_J(T^{l,l})$ and $\pi_K(T^{l,l})$ are H-minimal in $\mathbb{C} \mathrm{P}^{2n+1}$ if and only if $\pi(T^{l,0})$ is H-minimal in $\mathbb{C} \mathrm{P}^n$.
\end{example}

\begin{example}
  Let $T^{2,2}$ be the torus in $\mathbb{H}^4$ defined by the mapping
  \[
    \psi(x,y) = \frac{1}{2} (e^{jx_1} e^{ky_1}, e^{jx_2} e^{ky_2},
    e^{-jx_1} e^{-k(2y_1+y_2)}, e^{-jx_2} e^{ky_1}).
  \]
  The tori $\pi_J(T^{2,2})$ and $\pi_K(T^{2,2})$ are H-minimal, but not minimal in $\mathbb{C} \mathrm{P}^7$. Still, the torus $\pi(T^{2,0})$ is minimal in $\mathbb{C} \mathrm{P}^3$, and the torus $\pi(T^{0,2})$ is H-minimal, but not minimal in $\mathbb{C} \mathrm{P}^3$.
\end{example}

\begin{example}
  Let $T^{2,3}$ be the torus in $\mathbb{H}^4$ defined by the mapping
  \[
    \psi(x,y) = \frac{1}{2}(e^{jx_1} e^{ky_1}, e^{jx_2} e^{ky_2},
    e^{-jx_1} e^{ky_3}, e^{-jx_2} e^{-k(y_1+y_2+y_3)}).
  \]
  The tori $\pi(T^{2,0})$ and $\pi(T^{0,3})$ are minimal in $\mathbb{C} \mathrm{P}^3$ by \cite[Corollary~1]{Ovc18}. Then the minimality of $\pi_J(T^{2,3})$ and $\pi_K(T^{2,3})$ in $\mathbb{C} \mathrm{P}^7$ follows from Proposition \ref{prop:minimality}.
\end{example}


\begin{bibdiv}
  \begin{biblist}
    \bib{Bes08}{book}{
      author={Besse, A.},      
      title={Einstein Manifolds},
      publisher={Springer},
      address={Berlin},
      date={2008}}

    \bib{CM94}{article}{
      author={Chen, B.},
      author={Morvan, J.-M.},
      title={Deformations of isotropic submanifolds in K\"ahler manifolds},
      journal={Journal of Geometry and Physics},
      number={1},
      volume={13},
      date={1994},  
      pages={79--104}}
    
    \bib{CU98}{article}{
      author={Castro, I.},
      author={Urbano, F.},
      title={Examples of unstable Hamiltonian-minimal Lagrangian tori in $\mathbb{C}^2$},
      journal={Compositio Mathematica},
      number={1},
      volume={111},
      date={1998},      
      pages={1--14}}

    \bib{Law80}{book}{
      author={Lawson, H.~B.},      
      title={Lectures on minimal submanifolds},
      publisher={Publish or Perish},
      address={Berkeley},
      date={1980}}

    \bib{Ma05}{article}{
      author={Ma, H.},
      title={Hamiltonian Stationary Lagrangian Surfaces in $\mathbb{C} \mathrm{P}^2$},
      journal={Annals of Global Analysis and Geometry},
      number={1},
      volume={27},
      date={2005},
      pages={1--16}}
    
    \bib{Mir03}{article}{
      author={Mironov, A.~E.},
      title={On Hamiltonian-minimal Lagrangian tori in $\mathbb{C} \mathrm{P}^2$},
      journal={Siberian Mathematical Journal},
      number={6},
      volume={44},
      date={2003},
      pages={1039--1042}}

    \bib{Mir04}{article}{
      author={Mironov, A.~E.},
      title={New examples of Hamilton-minimal and minimal Lagrangian
    manifolds in $\mathbb{C}^n$ and $\mathbb{C} \mathrm{P}^n$},
      journal={Sbornik: Mathematics},
      number={1},
      volume={195},
      date={2004},
      pages={85--96}}

    \bib{MP13}{article}{
      author={Mironov, A.~E.},
      author={Panov, T.~E.},      
      title={Intersections of quadrics, moment-angle manifolds, and
    Hamiltonian-minimal Lagrangian embeddings},
      journal={Functional Analysis and Its Applications},
      number={1},
      volume={47},
      date={2013},
      pages={38--49}}
    
    \bib{Oh93}{article}{
      author={Oh, Y.-G.},
      title={Volume minimization of Lagrangian submanifolds under Hamiltonian deformations},
      journal={Mathematische Zeitschrift},
      number={1},
      volume={212},
      date={1993},      
      pages={175--192}}

    \bib{Ovc18}{article}{
      author={Ovcharenko, M.~A.},      
      title={On Hamiltonian-minimal isotropic homogeneous tori in
    $\mathbb{C}^n$ and $\mathbb{C} \mathrm{P}^n$},
      journal={Siberian Mathematical Journal},
      number={5},
      volume={59},
      date={2018},
      pages={931--937}}    
    
    \bib{HR00}{article}{
      author={Helein, F.},
      author={Romon, P.},
      title={Weierstrass representation of Lagrangian surfaces in four-dimensional space using spinors and quaternions},
      journal={Commentarii Mathematici Helvetici},
      number={4},
      volume={75},
      date={2000},  
      pages={668--680}}

    \bib{HR02}{article}{
      author={Helein, F.},
      author={Romon, P.},
      title={Hamiltonian stationary Lagrangian surfaces in $\mathbb{C}^2$},
      journal={Communications in Analysis and Geometry},
      number={1},
      volume={10},
      date={2002},  
      pages={79--126}}
    
    \bib{Sha91}{article}{
      author={Sharipov, R.~A.},      
      title={Minimal tori in the five-dimensional sphere in $\mathbb{C}^3$},
      journal={Theoretical and Mathematical Physics},
      number={1},
      volume={87},
      date={1991},
      pages={363--369}}

    \bib{Yer18}{article}{
      author={Yermentay, M.~S.},      
      title={On minimal isotropic tori in $\mathbb{C} \mathrm{P}^3$},
      journal={Siberian Mathematical Journal},
      number={3},
      volume={59},
      date={2018},
      pages={415--419}}

    \bib{Yer19}{article}{
      author={Yermentay, M.~S.},      
      title={On a family of minimal isotropic tori and Klein bottles in $\mathbb{C} \mathrm{P}^3$},
      journal={Siberian Electronic Mathematical Reports},
      number={},
      volume={16},
      date={2019},
      pages={955--958}}
  \end{biblist}
\end{bibdiv}
\end{document}